\documentclass{amsart}
\usepackage{graphicx,amsmath,amssymb,amsfonts,amsthm}
\usepackage[all,cmtip]{xy}
\usepackage[colorlinks=false,hyperindex]{hyperref}
\usepackage{color}
\usepackage{fullpage}
\usepackage{tikz}
\usepackage{tikz-cd}
\usepackage{url}

\newtheorem{theorem}{Theorem}[section]

\newtheorem{corollary}[theorem]{Corollary}
\newtheorem{proposition}[theorem]{Proposition}

\theoremstyle{definition}
\newtheorem{definition}[theorem]{Definition}

\newtheorem{conjecture}[theorem]{Conjecture}
\newtheorem{example}[theorem]{Example}

\theoremstyle{remark}
\newtheorem{remark}[theorem]{Remark}

\newcommand{\ZZ}{\mathbb{Z}}

\newcommand{\CC}{\mathbb{C}}
\newcommand{\RR}{\mathbb{R}}

\newcommand{\Krull}{\operatorname{dim}}
\newcommand{\trdeg}{\operatorname{trdeg}}
\newcommand{\ord}{\operatorname{ord}}

\renewcommand{\AA}{\mathbb{A}}

\newcommand{\Spec}{\operatorname{Spec}}

\newcommand{\Ocal}{\mathcal{O}}

\newcommand{\Frac}{\operatorname{Frac}}

\newcommand{\rk}{\operatorname{rk}}

\newcommand{\sat}{\operatorname{sat}}

\newcommand{\Max}{\operatorname{max}}
\newcommand{\red}{\operatorname{red}}
\newcommand{\Sym}{\operatorname{Sym}}

\newcommand{\len}{\operatorname{len}}

\newcommand{\taylor}[1]{{\color{blue} \sf $\clubsuit\clubsuit\clubsuit$ Taylor: [#1]}}

\newcommand{\defi}[1]{\textsf{#1}} 				

\begin{document}

\title{The Jacobi Bound Conjecture for Generically Reduced Differential Schemes}
\author{Taylor Dupuy and David Zureick-Brown}
\date{\today}
\maketitle
\begin{abstract}
 We prove the Strong Jacobi Bound Conjecture for generically reduced components of differential schemes.
\end{abstract}
\tableofcontents

\section{Introduction}

Since the Jacobi Bound Conjecture is analogous to B\'ezout's inequality for affine space we first recall this theorem. 
In what follows $\len_K$ denotes the length of an affine scheme which is the length of its underlying coordinate ring. 
\begin{theorem}
 Let $K$ be a field.
 Let $I =\langle f_1,\ldots,f_n\rangle \subset K[x_1,\ldots,x_n]$ be an ideal with $f_i$ of degree $d_i$ and let $X = \Spec K[x_1,\ldots,x_n]/I$.
 If $X$ has finite length, then $\len_K(X)\leq d_1d_2\cdots d_n$. 
\end{theorem}

Jacobi in 1865 \cite{Jacobi1865,Jacobi2009} stated an analog for solutions of algebraic differential equations.
The main trade-off is that `finite' in B\'ezout is replaced by `finite dimensional' in the Jacobi Bound Conjecture and the length is replaced by the dimension of the intersection. 
Later, Ritt in 1935 \cite{Ritt1935} realized that Jacobi's statement had no rigorous definition of ``dimension'', understood by Jacobi to be the number of parameters required to parametrize solutions of differential equations.\footnote{Jacobi himself used the word ``order'' for this undefined notion.}
Ritt remedied this by providing a definition based on transcendence degrees, declared it an open problem, and proved several cases of the conjecture.
The general case of this statement has been open since Ritt's original paper.
\\

In the present paper we prove the Strong Jacobi Bound Conjecture (Conjecture~\ref{C:jbc}) for all components which are generically reduced. See Section \ref{section:notation} for notation.

\begin{theorem}\label{T:main}
Let $(K,\partial)$ be a differential field of characteristic zero.
 Let $I = [u_1,\ldots,u_n] \subseteq K\lbrace x_1,\ldots,x_n\rbrace$ be a differential ideal, and let $\Sigma = \Spec K\lbrace x_1,\ldots,x_n\rbrace/I$. Let $\Sigma_1$ be an irreducible component of $\Sigma$ such that $\Sigma_1$ is generically reduced and has finite dimension. Then $\Krull(\Sigma_1)\leq J(u_1,\ldots,u_n)$. 
\end{theorem}

\begin{remark}
The number $J(u_1,\ldots,u_n)$ is the \defi{Jacobi number} or \defi{Jacobi bound} (Definition \ref{definition:jacobi-number}).
\end{remark}

\begin{remark}
If $I$ is radical (i.e.~$\sqrt{I} = I$), then every component of $\Sigma$ is reduced.
\end{remark}

The strategy is the following: if $\Sigma_1$ is generically reduced with generic point $\eta_1$, then $\Krull(T\Sigma_{\eta_1}) = \Krull(\Sigma_1)$ and the tangent space at $\eta_1$, $T\Sigma_{\eta_1}$, is defined by equations which are linearizations of the $u_1,\ldots,u_n$. 
The Jacobi bound of the linearizations is less than or equal to the original Jacobi bound. Since the Jacobi Bound Conjecture in the linear case is a Theorem (see \cite [pg 310]{Ritt1935} and \cite [Section 7]{Dupuy2026}), our result follows. We expand on these three sentences in Section~\ref{S:strategy}.

The contrapositive of the previous paragraph is also interesting; it says that whenever we have a bad linearization that we must be able to witness some non-reducedness. 
In particular, the remaining difficulty of the Jacobi Bound Conjecture arises precisely from non-reducedness. 
\begin{corollary}\label{C:contrapositive}
	Let $\Sigma \subset \AA^n$ be the $D$-scheme cut out by $u_1,\ldots,u_n \in K\lbrace x_1,\ldots,x_n\rbrace$.
	Let $\Sigma_1$ be an irreducible component of $\Sigma$ with $\dim(\Sigma_1)<\infty$. Let $\eta_1$ be the generic point of $\Sigma_1$. 
	If $\Krull((T\Sigma)_{\eta_1})> \Krull(\Sigma_1)$, then $\Sigma_1$ is not generically reduced. 
\end{corollary}
\begin{proof}
	Suppose it were reduced. 
	Then we know that the equations $u_1,\ldots,u_n$ defining $I_{\Sigma}=[u_1,\ldots,u_n]$ can be linearized at $\eta_1$ to give equations for $(T\Sigma_1)_{\eta_1}$. In other words $(T\Sigma_1)_{\eta_1} = (T\Sigma)_{\eta_1}$.
	This then implies $\Krull((T\Sigma_1)_{\eta_1})\neq \Krull(T\Sigma_{\eta_1})$, which is a contradiction.
\end{proof}
 
Some more remarks about the statement are in order. 

\begin{remark}
		The characteristic zero assumption seems to be essential for our proof as it is related to generic smoothness, which is closely related to Bertini's Theorem.
\end{remark}

\begin{remark}
There is a long history of partial progress toward the Jacobi Bound Conjecture over the last 100 years which is surveyed in \cite{Kondratrieva2008}.
In 1935 Ritt first declared Jacobi's proof incomplete and proved both the linear and two variable cases as well as giving a general bound $\Krull(\Sigma) \leq r_1 + \cdots + r_n$ where $r_i = \Max \lbrace \ord_{x_i}^{\partial}(u_j) \colon 1 \leq j \leq n \rbrace$ \cite{Ritt1935}.
In 1958, in his thesis, Greenspan \cite{Greenspan1958} improved upon Ritt's bound in the difference case (which could be transported to the differential case \cite[section 8.1]{Ollivier2009}).
See also the accompanying paper \cite{Greenspan1959}.
In 1969, Lando in her thesis studied the conjecture and its difference analog \cite{Lando1969}.
This work resulted in the papers \cite{Lando1970, Lando1972} in which she proved the first order cases of the ``strong conjecture'' as well as the analog of Ritt's statements in the difference case. 
In 1976 Tomasavic \cite{Tomasovic1976} introduced a version of the conjecture for partial differential equations and proved it in the order two and linear cases.
In 1980 Cohn gave the so-called Greenspan bound \cite{Cohn1980} for the Krull dimension.
In 2004 Hrushovski posted a manuscript on the arxiv proving the difference case using so-called uniform Lang--Weil estimates \cite{Hrushovski2004}.
In 2007 Ollivier and Sadik \cite{Ollivier2007} and then later Ollivier in 2022 \cite{Ollivier2022} proved the conjecture in the so-called quasi-regular case.
\end{remark}

There are certain instances where we get equality of dimension and the Jacobi bound. This is the case of characteristic sets. 
In particular this improves \cite[Theorem 1]{Li2016} which is ``the Jacobi bound conjecture for characteristic sequences''.
In this case we can upgrade the inequality to an equality. 
The inequality was proved in \cite[Theorem 1]{Li2016} using differential Chow forms. 

\begin{corollary}\label{corollary:characteristic-sets}
	Let $R=K\lbrace x_1,\ldots,x_n\rbrace$.
	Let $P \subset R$ be a prime differential ideal with characteristic set $(A_1,\ldots,A_n)$ with respect to some ranking $\prec$.
	Then
	$$ \trdeg_K(R/P) = J(A_1,\ldots,A_n). $$
	In particular $\trdeg_K(R/P)<\infty.$
\end{corollary}
The point is that $P=\sat_S([A_1,\ldots,A_n])$ where $S$ is the multiplicative set generated by initials and separants of $(A_1,\ldots,A_n)$ and that the initials and separants are not in $P$. 
If $\eta$ is the generic point of $P$, then this implies that $A_1,\ldots,A_n$ have a good linearization at $\eta$ since $\dfrac{\partial A_i}{\partial \ell_{A_i}}(\eta) \neq 0$.
The important equation for linearization in the text is \eqref{equation:linearized}.
By ``good linearization'' we mean the dimension of the component for $\eta$ is equal to the dimension of the linearized system at $\eta$.
See also \cite[Section 4.1]{Cluzeau2003}.

Corollary~\ref{corollary:characteristic-sets} implies that for any given explicit set of differential equations $u_1,\ldots,u_n \in K\lbrace x_1,\ldots,x_n\rbrace$ one can effectively determine whether Jacobi Bound Conjecture holds for $[u_1,\ldots,u_n]$ by a computer algorithm.
\begin{corollary}\label{corollary:decidability}
	The problem of determining if the Jacobi Bound Conjecture holds for a given set of differential equations $(u_1,\ldots,u_n)$ in $K\lbrace x_1,\ldots,x_n\rbrace$ is decidable. 
\end{corollary}
The recipe to check JBC in any example is simple. 
The Jacobi bound of our original equations is straightforward to compute. 
One then computes the characteristic sequences for the prime ideals using the Rosenfeld--Gr\"{o}bner algorithm.
One can compute the dimension of these using Corollary~\ref{corollary:characteristic-sets}.
	
From Corollary~\ref{corollary:decidability} it is hence desirable to determine what happens to the Jacobi bound during the Rosenfeld--Gr\"{o}bner algorithm. 
We report that the weak Jacobi bound can increase during steps of the Rosenfeld--Gr\"{o}bner algorithm (even for linear differential equations) but in the end always decreases back to the original Jacobi bound. 

In \cite{Dupuy2026} we give a more detailed analysis of the reduction process and its impact on the strong Jacobi bound (this is where we define $\ord_{x_i}^{\partial}(u)=-\infty$ if $x_i$ does not appear in $u$). 
In this process \cite{Dupuy2026} actually repairs a gap in Ritt's proof of the linear case in \cite{Ritt1935}.

\subsection{Strategy of Proof}\label{S:strategy}
Our approach is inspired by the excellent paper \cite{Ollivier2022}, which first translates Jacobi's work into the setting of modern algebra and then improves upon it.
Our desire to understand the geometric framework of \cite{Ollivier2022}  (in particular, the relationship between first order perturbations of nonlinear differential equations and the tangent bundle of differential schemes) evolved into the current manuscript.
\\

\noindent First, we observe that linearizing a system $[u_1,\ldots,u_n]$ at a known solution $\eta$ doesn't increase the Jacobi bound. 
This is the standard linearization technique taught in differential equations courses where given a solution $\eta$ we expand our equations $u(x)=0$ using $x = \eta+ \varepsilon y$ where $\varepsilon^2=0$, resulting in linear equations in the variable $y$ with coefficients involving $\eta$. 
This means that if we can control the Krull dimension of linearized systems we can prove the Jacobi Bound Conjecture. 
\\

\noindent Second, our key observation is that the tangent bundle $T\Sigma \to \Sigma$ really encodes these first order perturbations. 
This means that $\pi\colon T\Sigma \to \Sigma$ is really a parameter space for linearized differential equations $T\Sigma_{\eta} =\pi^{-1}(\eta)$ as we vary $\eta \in \Sigma$.
Many of the ingredients here are not new but perhaps not synthesized in quite this way before. 
The equations for linearizations in relation to the tangent spaces appear in Kolchin \cite{Kolchin1986} (see Remark~\ref{remark:other-tangent-spaces}) and $D$-module structures on Kahler differentials necessary for taking the modern scheme theory approach (where $T\Sigma := \underline{\Spec}(\Sym(\Omega_{\Sigma/K}))$) appear in \cite{Johnson1974} and \cite{Ollivier2022}. \\

\noindent Finally, many standard theorems about dimensions of schemes which are finite type over a field (e.g., the standard fact that the dimension of the tangent bundle at the generic point of a component is at least the dimension of the component) don't apply directly to a generically reduced but non-finite type component $\Sigma_1$. 
We need to show that $\Sigma_1$ is birational to a scheme of finite type. 
In general $\Sigma_1$ is not finite type and this only works birationally. We then need to be careful that we can relate dimension estimates to the tangent space of an open set back to our original equations.
Some extra care is need to apply theorem like ``generic smoothness'' in our setting.

\subsection*{Acknowledgments}
Dupuy is supported by National Science Foundation grant DMS-2401570.
Zureick-Brown is supported by National Science Foundation grant DMS-2430098.
We would like to thank James Freitag for useful comments on an earlier version of this manuscript. 
We would also like to thank Yuval Dor for helpful conversations.
Dupuy is grateful for the hospitality of the GAATI laboratory at the l'Université de la Polynésie française, where he completed this manuscript.
\\

\subsection*{Software}
Experiments for this manuscript were performed in Magma \cite{Magma} using a differential algebra package \cite{dupuyZB:magma-package} developed by the authors. 
No LLMs were used during this project.

\section{Background and Notation.}
 
\label{section:notation}
In this section we follow \cite[section 2]{Buium1993} for the basic setup of modern differential algebraic geometry; see also \cite{Kolchin1973} and \cite[Chapter VII]{Kaplansky1976}.

The number $\Krull(\Sigma_1)$ appearing in the statement is the dimension of the underlying scheme $\Sigma_1$.	
By the dimension of a scheme $X$ over a field $K$ we mean the maximum transcendence degree of $\kappa(X_1)$ as $X_1$ varies over irreducible components of $X$.
 By $\kappa(X_1)$ we mean the function field of the underlying reduced irreducible scheme of $X_1$. 
 If $X=\Spec(A)$ and $X_1 = V(Q)$ with $Q$ a $P$-primary ideal, then $\kappa(X_1)=\kappa(P)$ which is by definition the fraction field of $A/P$. 
 We will denote the dimension of $X$ (resp $A$) by $\dim(X)$ (resp $\dim(A)$).
 When $X$ or $A$ are finite type over $K$ this coincides with the usual noetherian dimension (resp Krull dimension) of the scheme (resp ring). 

In \cite[section 2, pg 485]{Buium1993} Buium called this the classical schematic dimension the ``absolute dimension'' so that it is not confused with the notion of differential dimension. 

\subsection{Differential Algebra Background}

\begin{definition}
A \defi{differential ring} is a pair $(R,\partial)$ where $R$ is a ring and $\partial\colon R \to R$ is an additive map satisfying $\partial(ab) = \partial(a) b + a\partial(b)$.	
A morphism of differential rings $(B,\partial)\to (A,\partial)$ is a ring homomorphism $\sigma$ such that $\partial \sigma = \sigma \partial$. 
Kernels of morphisms of differential rings are differential ideals and there is a first isomorphism theorem. 
\end{definition}
We will often omit the pair $(R,\partial)$ and just write $R$ for a differential ring.
Similarly we will just write $\sigma\colon B\to A$ for differential ring homomorphisms. 
If $A$ and $B$ are differential rings, by a differential $B$-algebra we mean a $B$-algebra $A$ such that the structure map $B\to A$ is a morphism of differential rings. 

\begin{definition}
Let $R$ be a differential ring. 
The ring of \defi{differential polynomials over $R$} on $n$-variables is the polynomial ring in countably many variables over $R$
$$
 R\lbrace x_1,\ldots,x_n\rbrace= R[x_i^{(r)} \colon 1\leq i \leq n, r \geq 0 ]
$$
with the unique differential $R$-algebra structure specified by $\partial x_i^{(r)} = x_i^{(r+1)}$ for $1\leq i \leq n$ and $r\geq 0$. 
\end{definition}
In this $R\lbrace x_1,\ldots,x_n\rbrace$ we identify $x_i^{(0)}$ with $x_i$ and use primes for low order derivatives as is customary, e.g.~$x_1'' = x_1^{(2)}$.

The order of a differential polynomial $u$ with respect to some differential variable is the largest derivative of the variable that appears. 
In the above notation given a finite set $S \subset \RR$ we let $\Max^{+}(S) = \Max( S \cup \lbrace 0 \rbrace)$. 
\begin{definition}
	Let $u\in R\lbrace x_1,\ldots,x_n \rbrace$.
	The \defi{order of $u$ with respect to} $x_j$ is 
	 \begin{equation}\label{equation:order}
	 	\ord_{x_j}^{\partial}(u) = \operatorname{max}^{+} \lbrace r \in \ZZ_{\geq 0} \colon \frac{\partial u}{\partial x_j^{(r)}} \neq 0\rbrace. 
	 \end{equation} 
\end{definition}
Given a collection of elements $u_1,\ldots, u_e$ in $R\lbrace x_1,\ldots,x_n\rbrace$ we use the notation 
$$[u_1,\ldots,u_e] = \langle \partial^j(u_i) \colon 1\leq i \leq e, j \geq 0 \rangle, $$
generated by the $u_i$'s and all of their derivatives.
Equivalently, it is the smallest ideal in $R\lbrace x_1,\ldots,x_n\rbrace$ which is closed under $\partial$ and contains the set $\lbrace u_1,\ldots,u_r\rbrace$.\footnote{
In differential algebra it is common to let $\lbrace u_1,\ldots,u_e\rbrace$ to denote the radical of $[u_1,\ldots,u_e]$. We do not follow this convention here and will reserve braces for sets.}

The following is well-known. See for example \cite[\S5, last paragraph]{Johnson1969a}. Since we don't know of a reference with a proof we supply one.
\begin{proposition}[Minimal Prime Principle]\label{proposition:minimal-prime-principle}
	Let $R = K\lbrace x_1,\ldots,x_n\rbrace$ and let $I \subset R$ be a differential ideal.
	\begin{enumerate}
		\item Any prime ideal $P \supset I$ which is minimal over $I$ is a differential ideal. 
		\item There exists a finite irredundant collection of differential primes $Q_1,\ldots,Q_m$ such that $\sqrt{I} = Q_1 \cap \cdots \cap Q_m$. 
	\end{enumerate}
\end{proposition}
\begin{proof}
	The second statement is \cite[Theorem 3.3 of Chapter 2]{Buium1994}. There exist distinct prime differential ideals $Q_1,\ldots,Q_m$ such that $I = Q_1 \cap \cdots \cap Q_m$. 
	Choosing $m$ minimally, it follows from \cite[Theorems 4.5 and 4.6]{Atiyah1969} that the $Q_i$ are exactly the set of minimal primes over $I$; in particular, every minimal prime over a differential ideal is a differential prime ideal.
\end{proof}

Our division algorithms will be phrased in terms of Weyl algebras.
\begin{definition}
Let $R$ be a differential ring. 
The \defi{Weyl algebra} of $R$ is the (non-commutative) associative algebra $R[\partial]$ which is isomorphic to $\bigoplus R\partial^d$ as a left $R$-module and satisfies the rule $\partial r = r \partial + \partial(r)$ for every $r \in R$.
\end{definition}

\begin{definition}
	A left $R$-module $M$ is a \defi{$D$-module} if it is a left $R[\partial]$-module.
\end{definition}
A left $R$-module $M$ being a $D$-module is equivalent to being equipped with a map $\partial\colon M \to M$ such that $\partial(rm) = \partial(r)m + r\partial(m)$ for $r\in R$ and $m\in M$.
The differential ring $R$ itself is a $D$-module since it is a left $R[\partial]$-module. 
A differential ideal $[u_1,\ldots,u_e]$ for $u_1,\ldots,u_e \in R\lbrace x_1,\ldots,x_n\rbrace$ is a $D$-module which is generated as a left $R\lbrace x_1,\ldots,x_n\rbrace[\partial]$-module by $\{u_1,\ldots,u_e\}$; $[u_1,\ldots,u_n] = R\lbrace x_1,\ldots,x_n\rbrace[\partial] u_1 + \cdots + R\lbrace x_1,\ldots,x_n\rbrace[\partial] u_n$. It is the smallest left $R\lbrace x_1,\ldots,x_n\rbrace[\partial]$-module of $R\lbrace x_1,\ldots,x_n\rbrace$ containing $u_1,\ldots,u_e$ and every element $f\in [u_1,\ldots,u_e]$ can be written $f=Q_1u_1 + \cdots Q_eu_e$ where $Q_i \in R\lbrace x_1,\ldots,x_n\rbrace[\partial]$ for $1\leq i \leq e$. 

For $D$-schemes and differential varieties we follow \cite[Ch 3, \S3, pg 68]{Buium1994}.
\begin{definition} 
	By \defi{$D$-scheme} we mean a scheme $\Sigma$ such that $\Ocal_{\Sigma}$ is a sheaf of differential rings. 
	A morphism of $D$-schemes is a morphism of schemes where the morphism of structure sheaves is a morphism of differential rings. 
\end{definition}
	We also call $D$-schemes \defi{differential schemes}.
 If $R$ is a differential ring, then $\Spec(R)$ is a $D$-scheme. 
 By a $D$-scheme over $R$ we mean a morphism of $D$-schemes $\Sigma \to \Spec(R)$. 
 Equivalently this is a $D$-scheme $\Sigma$ such that $\Ocal_\Sigma$ is a sheaf of differential $R$-algebras. 
 
 Let $\Sigma$ be a $D$-scheme over $K$. 
 By a \defi{differential $L$-point} of $\Sigma$ we mean a morphism $\Spec(L) \to \Sigma$ which is a morphism of $D$-schemes.
 	
 \begin{remark}
 	Suppose that $\Sigma = \Spec K\lbrace x_1,\ldots,x_n\rbrace/I$ where $I$ is a differential ideal. 
 	Not all scheme theoretic points $x \in \Sigma$ are differential points. 
 	The generic point $\eta$ of an irreducible component of $\Sigma$ is a differential point.
 	This follows from Proposition~\ref{proposition:minimal-prime-principle}.
 \end{remark}

 Let $K$ be a field. 
 For any $K$-algebra $R$ by $\Krull(R)$ we mean the largest $n$ such that there exists an injection of a polynomial ring in $n$ variables over $K$ into $R$. 
 Similarly for any scheme $\Sigma$ over $K$ we will let $\Krull(\Sigma)$ denote the dimension of $\Sigma$.
 For irreducible schemes $\Sigma$ this is $\trdeg_K(\kappa(\Sigma))$, the transcendence degree over $K$ of the function field $\kappa(\Sigma)$ of the underlying reduced scheme.
 For general schemes over $K$, $\dim(\Sigma)$ will mean maximal dimension of its irreducible components.
 For $D$-schemes, to distinguish this from the ``differential dimension'', this is called the \defi{absolute dimension} in \cite{Buium1993} and the \defi{order} in \cite{Kolchin1973}.
 We have chosen $\Krull$ as it is less ambiguous and agrees with the standard terminology used by commutative algebraists and algebraic geometers.
 If $\Sigma$ is finite type and integral over a field $K$, then $\Krull(\Sigma)$ coincides with the Krull dimension.

 \begin{definition}
 \label{definition:jacobi-number} 
	Let $R$ be a differential ring. 
	Let $u_1,\ldots,u_n \in R\lbrace x_1,\ldots,x_n\rbrace$.
	The \defi{Jacobi number} of the tuple $(u_1,\ldots,u_n)$ is defined to be 
	$$J(u_1,u_2,\ldots,u_n) = \Max_{\sigma \in S_n} \sum_{i=1}^n \ord_{x_i}^{\partial}(u_{\sigma(i)}).$$
\end{definition}

We now state the Jacobi bound conjecture. 
 \begin{conjecture}[Jacobi Bound Conjecture (JBC) {\cite[pg 303]{Ritt1935}}, {\cite[pg 274]{Kolchin1968}}]\label{C:jbc}
Let $K$ be a differential field and let $\Sigma$ be a differential algebraic variety defined by $u_1,u_2,\ldots,u_n \in K\lbrace x_1,\ldots,x_n\rbrace$. Let $\Sigma_1$ be an irreducible component of $\Sigma$ and suppose that $\Sigma_1$ has finite Krull dimension. Then
\[
 \Krull(\Sigma_1) \leq J(u_1,u_2,\ldots,u_n).
\]
 \end{conjecture} 

\subsection{Ritt's Division Algorithm and Characteristic Sets}
\label{subsection:pseudodivision-and-characteristic-sets}
Let $R$ be a differential ring. 
\begin{definition}
	A \defi{ranking} on $R\lbrace x_1,\ldots,x_n\rbrace$ is an ordering $\prec$ on the set $\lbrace x_i^{(j)} \colon 1 \leq i \leq n, j \geq 0 \rbrace$ such that for all $v$ and $u$ in that set if $u\prec v$, then $\partial(u) \prec \partial(v)$ and such that for all $u$ in the set $u \prec \partial(u)$. 
\end{definition}
We identify rankings with the lexicographic term orders they induce on $R\lbrace x_1,\ldots,x_n\rbrace$. 
Often we use a \defi{Ritt ranking} where terms are ordered by variable, then by order, then by degree. 
\begin{definition}
	Fix a ranking $\prec$ on $R\lbrace x_1,\ldots,x_n \rbrace$. 
	Let $A \in R\lbrace x_1,\ldots,x_n\rbrace$.
	\begin{enumerate}
		\item The \defi{leader} $\ell_A$ of $A$ is the largest variable occurring in $A$ according to the ranking. 	
		\item If $\ell_A = x_i^{(r)}$, then the \defi{leading variable} is $t_A = x_i.$
		\item The \defi{separant} of $A$ to be $s_A=\dfrac{\partial A}{\partial \ell_A}$.
		\item If $A = \sum_{i = 0}^d a_i\ell_A^i$, then the \defi{initial} $I_A$ of $A$ is $a_d$.
	\end{enumerate}
\end{definition}
The \defi{rank} of an element $A = \sum_{i = 0}^{d_A} a_i\ell_A^i\in R\lbrace x_1,\ldots,x_n\rbrace$ is the $\ell_A^{d_A}$.
We say that $B$ is \defi{Ritt-reduced} with respect to $A$ if $B$ is of lower rank in the leader of $A$.
We note that for $A \in K\lbrace x_1,\ldots,x_n\rbrace$ and $j > 0$, we have  
$$
I_{A^{(j)}} = s_A\ell_{A}^{(j)}.
$$
In other words, while differentiating generally jumbles around the terms of a polynomial $A$, there is a nice formula for the leading term of any derivative of $A$.

Now let's work with the polynomial ring $K\lbrace x_1,\ldots,x_n\rbrace$ with coefficients in a differential field $K$ of characteristic zero; all fields we consider will have characteristic zero.

\begin{example}
	Consider $K\lbrace x_1,x_2\rbrace$ with $x_1\prec x_2$. 
	Consider $A,B \in K\lbrace x_1,x_2 \rbrace$ where $A = x_1'+1$ and $B = x_2''x_1''+x_1^2$. 
	Although $A\prec B$, $B$ is not reduced with respect to $A$ since $x_1''$ appears in $B$ and $x_1'' \succ x_1' = \ell_{A}$.
	The separant and initial of $C=x_1x_1'+1$ are both $x_1$.
\end{example}

This modified division algorithm (where we take derivatives, and multiply by the separant) is called \defi{Ritt Division}.

\begin{theorem}[Ritt's Division Algorithm {\cite[Lemma 7.3]{Kaplansky1976}}]
 \label{theorem:pseudodivision}
	Let $A, B \in R = K\lbrace x_1,x_2,\ldots,x_n \rbrace$. Then there exists some $s,t\in \ZZ_{\geq 0}$, $Q \in R[\partial]$, and $C \in R$ such that 
\[
	 S_A^s I^t_A B = QA + C
\]
where $C$ is reduced with respect to $A$ and the operation of $Q$ is via the usual $D$-module structure.
\end{theorem}

Let $A = (A_1, \ldots, A_r)$ be a sequence of elements of $K\lbrace x_1,\ldots,x_n \rbrace$. We say that $A$ is \defi{autoreduced} if each $A_{i+j}$ is reduced with respect to $A_i$, and $A_{i+1}$ involves a variable that does not appear in $A_i$.
Let $S=S_A$ be the multiplicative monoid generated by the separants and initials of $A_1,\ldots,A_r$. Let $B \in R$. If $A$ is autoreduced, then repeated application of Theorem \ref{theorem:pseudodivision} gives an element $s \in S$, $Q_i \in R[\partial]$, and $c \in R$ such that
\[
sB = Q_1A_1 + \cdots + Q_rA_r + c
\]
where $c$ is reduced with respect to each $A_i$.
This is called the \defi{Ritt division algorithm}.
For details we refer to \cite{Hubert2000}.

\begin{example}
Let $f=x'+y'''$ and $g=x^2+y''x'+t$ in $\CC(t)\lbrace x,y\rbrace$ equipped with a term order which is lexicographic with $x^{(i)}>y^{(j)}$ for all $i$ and $j$, and $x^{(i+1)}>x^{(j)}$ (resp., $y^{(i+1)}>y^{(j)}$) for all $i$.
Then we have 
		$$ s f = Qg + r $$
where 
 $$ s= -(x')^2,\quad Q=x''-x'\partial, \quad r=-x^2x'' + -tx'' + -(x')^3 + 2x(x')^2 + x'.$$

\end{example}
One defines a partial ordering on autoreduced inductively. 
Let $A = (A_1,A_2,\ldots, A_r)$ and $B = (B_1,B_2,\ldots,B_s)$.
We say that $A \succ B$ if and only if
\begin{enumerate}
	\item there exists some $i$ with $0<i<r$ such that
	$$\ell_{A_j}^{d_{A_j}} = \ell_{B_j}^{d_{B_j}}\qquad \forall j, \mbox{ with } 0<j<i$$ where $d_{A_j} = \deg_{\ell_{A_j}}(A_j)$ and $d_{B_j}=\deg_{\ell_{B_j}}(B_j)$ and $A_i \prec B_i$, or
	\item 
	$\ell_{A_j}^{d_{A_j}} = \ell_{B_j}^{d_{B_j}}$ for all $j\leq r$ and $s>r$.
\end{enumerate}

Let $I \subseteq K\lbrace x_1,x_2,\ldots,x_n \rbrace$ be an ideal. 
\begin{definition}
	A sequence $A_1,\ldots,A_r \in I$ is a \defi{characteristic sequence} for $I$ if $A_1,\ldots,A_r$ is a minimal autoreduced sequence. 
\end{definition}
Characteristic sequences are often called \defi{characteristic sets} in the literature and we will sometimes do so but in doing so we keep in mind that these are really sequences. 

If $S$ is a multiplicatively closed subset of a ring $R$ and $I$ is an ideal in $R$ the \defi{saturation} of $I$ with respect to $R$ is $\sat_S(I) = \lbrace f \in R \colon \exists s \in S, sf \in I \rbrace$. 
This is the same as $\varphi_S^{-1}(I S^{-1}R)$ where $\varphi_S\colon R \to S^{-1}R$ is the localization map \cite[pg 41]{Atiyah1969}.

\begin{theorem}[{\cite[Theorem 4.5, Theorem 5.2, and Algorithm 7.1]{Hubert2000}}]
\label{theorem:characteristic-sets-exist}
 If $P$ is a prime $\partial$-ideal, then there exists a characteristic set $A=(A_1,\ldots,A_r)$ with $A_i \in P$ such that $P=\sat_S([A_1,\ldots,A_r])$ where $S$ is the multiplicative subset generated by the initials and separants of $A$.
\end{theorem}

\begin{remark}
 In any characteristic sequence the leading variables of the $A_i$ are all distinct, otherwise the sequence would not be autoreduced. 
Let $P$ be a prime differential ideal with characteristic set $A=(A_1,\ldots,A_r)$. If the differential scheme $\Spec K\lbrace x_1,x_2,\ldots,x_n \rbrace/P$ has finite Krull dimension, then it follows from the saturation claim that $r = n$. 
Otherwise the differential transcendence degree would be infinite since there would be a free differential variable not appearing as the leading variable of any of the elements. 
\end{remark}

For a prime ideal $P \subset K\lbrace x_1,\ldots,x_n\rbrace$ with a characteristic sequence $A=(A_1,\ldots,A_n)$ and an element $f \in K\lbrace x_1,\ldots,x_n\rbrace$ we have that $f \in P$ if and only if the remainder of $f$ after Ritt division by $A$ is zero. 
\cite[III.8, Lemma 5]{Kolchin1973}. 
This is called sometimes called Rosenfeld's Lemma. See \cite[below Theorem 1]{Kondratieva2006}.

\begin{example}
In the ring $K\lbrace x, y \rbrace$ with the ranking given by $((x),(y))$ we will compute the characteristic sets of $\sqrt{[f_1,f_2]}$ where $f_1=x''+y$ and $f_2=(x')^2+y$.
The differential ideal $\sqrt{I}$ has two components $P_1$ and $P_2$ where the first is given by 
 $$ P_1 = \sat_{S_1}([(y')^2+4y^3, 2yx'-y']),$$
where $S_1$ is the multiplicatively closed subset of $K\lbrace x,y\rbrace$ generated by $y,y'$ which are (up to a multiple of $K$) the separants of $2yx'-y$ and $(y')^2+4y^3$ respectively. 
The second component is given by 
 $$ P_2 = [y,x']$$
and is saturated by nothing (all the initials and separants here are trivial).
\end{example}

\section{Generic Smoothness}\label{section:generic-smoothness}
Our proof will rely on the fact that if $\Sigma_1$ is a generically reduced differential scheme of finite Krull dimension, then its tangent bundle is finite dimensional. If one omits the generic reducedness hypothesis, the tangent bundle of $\Sigma_1$ may fail to be finite dimensional, even if $\Sigma_1$ is itself finite dimensional; see Example \ref{example:really-big-point}.

\begin{example}
 \label{example:really-big-point}
 The ring $A = K\lbrace x\rbrace/[x^2]$ is not finitely generated. Its minimal prime is $[x] = \langle x,x^{(1)},\ldots,x^{(n)},\ldots\rangle$, the reduction $A_{\red}$ of $A$ is isomorphic to $K$, and for each $i$, $x^{(i)}$ is nilpotent.
 First observe that $\partial(x^2) = 2xx'$ and $\partial^2(x^2)=2(x')^2+2xx''$.
 Multiplying $\partial(x^2)$ by $x'$ implies that $(x')^3 \in [x^2]$. 
 We can then proceed inductively. 
 Indeed, by \cite[Lemma 1.7]{Kaplansky1976}, if $a^n \in [x^2]$, then $(a')^{2n-1} \in [x^2]$.
 This proves $(x'')^5 \in [x^2], (x''')^9 \in [x^2]$ etc. 
 In particular, $\Sigma = \Spec A$ is topologically a single, highly non-reduced point, and $\Sigma$ is not generically reduced.
\end{example}

\begin{proposition}
\label{prop:finite-typeness}
Let $\Sigma = \Spec K\{x_1,\ldots,x_n\}/[u_1,\ldots,u_n]$ and let $\Sigma_1$ be an irreducible component of $\Sigma$ such that $\Sigma_1$ is generically reduced and has finite Krull dimension.
There exist nonempty open subsets $U' \subseteq U \subseteq \Sigma_1$ such that $U$ is finite type over $K$ and $U'$ is smooth over $K$.
\end{proposition}
In general, it is not clear whether $\Sigma_1$ itself is locally of finite type; our proof produces a particular $U$ that is finite type. 
\begin{proof} 
 Let $R=K\{x_1,\ldots,x_n\}$. Let $P_1 \subseteq K\lbrace x_1,\ldots,x_n\rbrace$ be the prime ideal corresponding to $\Sigma_1$, and let $A_1 = K\lbrace x_1,\ldots,x_n \rbrace/P_1$ so that $\Spec(A_1) = (\Sigma_1)_{\text{red}}$.

 Let $g_1,\ldots,g_n \in P_1$ be a characteristic sequence for $P_1$ and let $m$ be the maximum of the orders of the $g_i$; note that since $\Sigma_1$ has finite Krull dimension, this set consists of $n$ distinct elements, and the remainder upon Ritt division has order at most $m$.
 Let $S$ be the ideal generated by the separants and initials of the characteristic set. 

 Let $K[x_1,\ldots,x_n]^{(m)} =K[x_i^{(j)} : 1\leq i \leq n, 0\leq j \leq m]$.
 We claim that the map $\sigma\colon S^{-1} K[x_1,\ldots,x_n]^{(m)} \to S^{-1}A_1$
 is surjective. It suffices to show that for every $j > m$, $x_i^{(j)}$ is in the image of $\sigma$.
 By Ritt division (see Subsection \ref{subsection:pseudodivision-and-characteristic-sets}), there are elements $s \in S$, $Q_i \in R[\partial] $, and $c \in R$ such that the order of $c$ is at most $m$ and
	\begin{equation*}
	s x_i^{(j)} = Q_1 g_1 + \cdots + Q_n g_n+c. 
 \end{equation*}
 Since the order of $c$ is at most $m$, $c$ is in the image of $\sigma$; since each $Q_i g_i \in P_1$, $Q_i g_i$ is zero in $A_1$, and we conclude that $s x_i^{(j)}$ is also in the image of $\sigma$. Since we have inverted $s$, $x_i^{(j)}$ is in the image of $\sigma$, as desired.
 
 Since $\sigma$ is surjective, $V = \Spec S^{-1}A_1$ is a nonempty open subset of $(\Sigma_1)_{\text{red}}$ finite type. Since $\Sigma_1$ is generically reduced, the map $i\colon (\Sigma_1)_{\text{red}} \to \Sigma_1$ is birational. Let $W \subseteq \Sigma_1$ be a nonempty open subset such that the restriction of $i$ to $W$ is an isomorphism. Then $U = W \cap V$ is the desired nonempty open subset of $\Sigma_1$ such that $U$ is finite type over $K$.
 
Since $U$ is integral and finite type, by \cite[Theorem 21.6.4]{vakil:rising-sea-published} $U$ is generically smooth over $K$, proving the claim about $U'$. 
\end{proof}

\begin{proposition}\label{proposition:local-dimension}
	Let $\Sigma = \Spec K\{x_1,\ldots,x_n\}/[u_1,\ldots,u_n]$ and let $\Sigma_1$ be an irreducible component of $\Sigma$ such that $\Sigma_1$ is generically reduced and has finite Krull dimension.
 Let $\eta_1 \in \Sigma_1$ be the generic point. 
Then $ \Krull(\Sigma_1) = \Krull( T\Sigma_{\eta_1}).$ 
\end{proposition}
\begin{proof}
 By Proposition \ref{prop:finite-typeness}, there exists a nonempty open subset $U \subseteq \Sigma_1$ such that $U$ is smooth over $K$. Then 
 \[
\Krull(\Sigma_1) = \Krull(U) = \Krull( TU_{\eta_1}) = \Krull( T(\Sigma_1)_{\eta_1}) = \Krull( T\Sigma_{\eta_1})
\]
where the second equality is true by smoothness of $U$ and the remaining properties about commutativity of various localization operations. 
\end{proof}

\section{Tangent Bundles of Differential Schemes and Linearization}
\label{section:tangent-bundle}
Let $\Sigma = \Spec K\lbrace x_1,\ldots,x_n\rbrace/[u_1,\ldots,u_n]$ and let $\Sigma_1$ be a generically reduced component. 
Since $\Sigma_1$ is generically smooth (Proposition~\ref{prop:finite-typeness}), its dimension is equal to the dimension of its tangent space at its generic point (Proposition~\ref{proposition:local-dimension}). 

In this section we prove that the tangent space $T\Sigma_{\eta}$  of $\Sigma$ at a differential point $\eta$ is isomorphic to the ``perturbation space'' $L[\Sigma,\eta]$ (i.e., the linearization) of our differential equations at $\eta$. 
Corollary~\ref{corollary:tangent-bundle-and-linearization} will show that these are isomorphic.
Moreover, we then show that the Jacobi number of the linear differential equations defining $L[\Sigma,\eta]$ are bounded by the original Jacobi number $J(u_1,\ldots,u_n)$ (Proposition~\ref{proposition:linearizing-jacobi-numbers}). 
Since Ritt proved the linear case of the Jacobi Bound Conjecture \cite{Ritt1935}, putting this all together (Section~\ref{section:main-proof}) will prove that $\Krull \Sigma_1$ is bounded by $J(u_1,\ldots,u_n)$, proving JBC for $\Sigma_1$.

\subsection{Dimension of Tangent Bundles}\label{section:dimension-of-tangent-bundles}
Let $I = [u_1,\ldots,u_n]$, with $u_1,\ldots,u_n\in R = K\{x_1,\ldots,x_n\}$, $A = R/I$, $\Sigma = \Spec A $, and let $\Sigma_1$ be an irreducible component of $\Sigma$ such that $\Sigma_1$ is generically reduced and has finite Krull dimension.
 By Proposition \ref{prop:finite-typeness} there exists a nonempty open set $U \subset \Sigma_1$ such that $U$ is smooth over $K$.

Let $\Omega_{A/K}$ be the $A$-module of K\"{a}hler differentials and let $\Omega^1_{\Sigma/K} = \widetilde{\Omega_{A/K}}$ be the sheaf of relative differentials. 
Since $U$ is smooth, $(\Omega^1_{\Sigma/K})|_U = \Omega^1_{U/K}$ is locally free, and since $U$ is smooth at $\eta$, by \cite[III.10]{Hartshorne1977},
\begin{equation}\label{E:component-dimension}
\Krull \Sigma_1 = \Krull U = \rk \Omega^1_{U/K}.
\end{equation}

 Next, $\Omega^1_{\Sigma/K}$ is a sheaf, but we can equivalently consider the ``physical'' tangent bundle $\pi\colon T\Sigma \to \Sigma$.
 \begin{definition}
 	The \defi{tangent bundle} $T\Sigma$ of $\Sigma$ is defined to be $T\Sigma = \underline{\Spec}(\Sym(\Omega_{A/K}))$
 	where $\underline{\Spec}$ is the global $\Spec$ functor and $\Sym$ is the symmetric algebra.
 \end{definition}

 Since $U$ is smooth over $K$ and $\Omega^1_{\Sigma/K}$ is locally free, by \cite[II, Exercise 5.18]{Hartshorne1977} for any $\eta \in U$,
 \begin{equation}\label{E:tangent-space-dimension}
\rk \Omega^1_{U/K} = \Krull \pi^{-1}(\eta).
 \end{equation}
 
	\begin{remark}\label{remark:other-tangent-spaces}
	Let $\Sigma$ be a $D$-scheme over a field $K$. 
	There are two other definitions of differential tangent spaces which we will not use.
	Both definitions are pointwise and since the important part for us is that the definition of differential tangent space has equations which are the first order perturbations of our non-linear differential equations (see Equation \eqref{equation:linearized}), these definitions are equivalent to ours at the points to which they apply.
	
	For a scheme $X$ the \defi{Zariski tangent space} at a point $x\in X$ is defined to be the $\kappa(x)$-vector space dual of $m_x/m^2_x$. 
	Kolchin takes the Zariski tangent space approach for the tangent space at $\Sigma$ and arrives at the same formulas \eqref{equation:linearized}. 
	See \cite[pg 198]{Kolchin1973}. This treatment also appears in \cite[p.5 before Fact 1.1]{Pillay2003} 
	We note that to give a point $\Spec(F[\varepsilon]/\langle \varepsilon^2 \rangle) \to X$ is equivalent to giving a point $x \in X(F)$ with $\kappa(x)=F$ and an element of $T_x$ \cite[Ch II, Problem 2.8]{Hartshorne1977}. 
	
	Yet another approach is to define the tangent space of $x$ at a point to be the derivations of $\Ocal_{X,x}$. 
	This is undertaken in \cite[Ch VII, Section 2, pg 197]{Kolchin1986}.
	At the bottom of \cite[p 198]{Kolchin1986} they arrive again at \eqref{equation:linearized} from this viewpoint. 
	The equations also appear in \cite[Ch V, Section 20, pg 334]{Kolchin1986}.
\end{remark}

\subsection{Linearizations of Schemes In Countably Many Variables}
In \ref{section:dimension-of-tangent-bundles} we defined the tangent bundle of $\Sigma\to \Spec(K)$ using a symmetric algebra. 
An alternative and equivalent way to define $T\Sigma$ is via first order perturbations of differential equations,
given at the functor of points level as 
$$ T\Sigma(B) = \Sigma(B[\varepsilon]/(\varepsilon^2),\partial) $$
where $B$ is a $K$-algebra.
We denote this functor by $L[\Sigma]$ and show in Proposition~\ref{proposition:isomorphism-with-symmetric-algebra} that it coincides with $T\Sigma$.

Switching temporarily from differential to commutative algebra, let $B = K[x_1,x_2,\ldots]$ be a polynomial ring with countably many variables, let $J \subseteq B$ be an ideal, let $C = B/J$, and let $X = \Spec C$. For $u \in J$, define $L[u]$ to be
 \begin{equation}
 \label{equation:linearization-definition}
 L[u]:=\sum_i \dfrac{\partial u}{\partial x_i}y_i \in C[y_1,y_2,\ldots]
 \end{equation}
where $C[y_1,y_2,\ldots]$ has a new variable $y_i$ for each previous variable $x_i$.
\begin{remark}
	Equation \ref{equation:linearization-definition} is what one obtains from plugging in $x + \varepsilon y$ into $u$, assuming that $\varepsilon^2=0$ and that $x$ is a solution of $u$, and then collecting the $\epsilon$ terms. 
	In other words, \ref{equation:linearization-definition} are the equations one obtains from first order perturbations. 
\end{remark}
For $\eta \in X$, corresponding to a prime ideal $P$, define $L[u,\eta]$ to be the image of $L[u]$ in the field $\kappa(\eta) = \Frac(B/P)$ (i.e., ``evaluate $L[u]$ at $\eta$''); this is equal to 
 \[
 L[u, \eta]:=\sum_i \dfrac{\partial u}{\partial x_i}(\eta)y_i \in \kappa(\eta)[y_1,y_2,\ldots].
 \]
 Define $L[J]$ to be the ideal $\langle L[u] : u \in J \rangle$, and define $L[J,\eta]$ to be the ideal $\langle L[u,\eta] : u \in J\rangle$.
 
 \begin{definition}
 	Let $B = K[x_1,x_2,\ldots]$ be a polynomial ring with countably many variables.
 	Let $J \subseteq B$ be an ideal and let $C = B/J$.
 	Let $X = \Spec(C)$. 
 	We define $L[X]$ to be $\Spec C [y_1,y_2,\ldots] / L[J]$, and $L[X,\eta]$ to be $\Spec \kappa(\eta) [y_1,y_2,\ldots] / L[J,\eta]$.
 \end{definition} 
 
 \begin{proposition}\label{proposition:isomorphism-with-symmetric-algebra}

 	Let $B = K[x_1,x_2,\ldots]$ be a polynomial ring with countably many variables.
 	Let $J \subseteq B$ be an ideal and let $C = B/J$.
 The map $dx_i^{(j)} \mapsto y_i^{(j)}$ defines an isomorphism $\Sym(\Omega_{C/K}) \to C [y_1,y_2,\ldots] / L[J]$ of $C$-algebras.
 \end{proposition}
 
 \begin{proof}
 The elements $dx_i^{(j)}$ and $y_i$ generate $\Sym(\Omega_{C/K})$ and $C [y_1,y_2,\ldots] / L[J]$ (respectively) as $C$-algebras.
 We claim that this map matches up the relations. Since $\Omega_{C/K} \cong \Omega_{B/K} / \Omega_{J/K}$
 \[
 \Omega_{C/K} \cong \Omega_{B/K} / \Omega_{J/K},
 \]
 it suffices to observe that for $u \in J$,
 \[
 du = \sum_i \dfrac{\partial u}{\partial x_i}dx_i \mapsto \sum_i \dfrac{\partial u}{\partial x_i}y_i = L(u).
 \]
\end{proof}

\subsection{Linearizations of $D$-Schemes}

Now consider a differential ring of the form $C = K\lbrace x_1,\ldots,x_n \rbrace/I$, i.e.~differential $(K,\partial)$-algebras which are $\partial$-finite type over $K$.
This gives the following.

 \begin{corollary}\label{corollary:tangent-bundle-and-linearization}
Let $\Sigma$ be an affine $D$-scheme over $\Spec(K)$ which is $\partial$-finite type. 
The schemes $T\Sigma \to L[\Sigma]$ are isomorphic as $\Sigma$-schemes: there exists a functorial isomorphism $T\Sigma \to L[\Sigma]$ such that diagram
\[
\xymatrix{
T\Sigma \ar[dr]_{\pi} \ar[rr] &&L[\Sigma] \ar[dl]^{\nu}\\
&\Sigma&}
\]
commutes. In particular, $T\Sigma_{\eta}=\pi^{-1}(\eta)\cong \nu^{-1}(\eta)=L[\Sigma]_{\eta}$ and $\Krull T\Sigma_{\eta} = \Krull L[\Sigma]_{\eta}$ for all $\eta \in \Sigma$. 
\end{corollary}

Let $\eta$ be a point of $\Sigma=\Spec(K\lbrace x_1,\ldots,x_n\rbrace/I)$. Then specializing gives an isomorphism
\begin{equation}\label{equation:linearization-at-point}
L[\Sigma]_{\eta} \cong \Spec \kappa(\eta)\{y_1,\ldots,y_n\} /L[I,\eta].
\end{equation}

We pause for a moment to record a special form of \eqref{equation:linearization-definition} in the case of differential polynomials and differential points. 
Let $u \in K[ x_1,\ldots,x_n]^{(r)}=K[x_i^{(j)} \colon 1\leq i \leq n, 0 \leq j \leq r]$. 
Write $x=(x_1,\ldots,x_n)$. 
Let $\eta \in \Sigma$ be a differential point. 
Let $y=(y_1,\ldots,y_n)$ be a new set of differential indeterminates. 
We have $$0=u(\eta+\varepsilon y) = u(\eta) + \varepsilon \dfrac{\partial u}{\partial \nabla^r x}(\eta) \cdot \nabla^r(y) = \varepsilon \dfrac{\partial u}{\partial \nabla^r x}(\eta) \cdot \nabla^r(y),$$
where 
\begin{align*}
	\nabla^r(x) &= (x_1,\dot{x}_1,\ldots, x_1^{(r)}, x_2, \dot{x}_2,\ldots, x_2^{(r)},\ldots, x_n,\dot{x}_n, \ldots, x_n^{(r)}),\\
	\dfrac{\partial u}{\partial \nabla^r x} &= \left( \dfrac{\partial u}{\partial x_1}, \dfrac{\partial u}{\partial \dot{x}_1},\ldots,\dfrac{\partial u}{\partial x^{(r)}_1 },\dfrac{\partial u}{\partial x_2}, \dfrac{\partial u}{\partial \dot{x}_2},\ldots,\dfrac{\partial u}{\partial x^{(r)}_2 },\ldots, \dfrac{\partial u}{\partial x_n}, \dfrac{\partial u}{\partial \dot{x}_n},\ldots,\dfrac{\partial u}{\partial x^{(r)}_n } \right),
\end{align*}
which implies the linearized equations take the form 
\begin{equation}\label{equation:linearized}
	\dfrac{\partial u}{\partial \nabla^r x}(\eta) \cdot \nabla^r(y) =0.
\end{equation}
We then have $L[u,\eta]:=\dfrac{\partial u}{\partial \nabla^r x}(\eta) \cdot \nabla^r(y)$.

We now look at linearizations of differential ideals. 
Let $\eta \in \Sigma =\Spec K\lbrace x_1,\ldots,x_n\rbrace/[u_1,\ldots,u_n]$ be a differential point.
Let $I = [u_1,\ldots,u_n]$ be the defining ideal of $\Sigma$.
We know that the linearization of the differential ideal $I$ is the differential ideal of the linearizations:
\begin{equation}\label{equation:linear-equations}
L[I,\eta] = [L[u_1,\eta],\ldots,L[u_n,\eta]].
\end{equation}
Indeed, $L[\Sigma]_{\eta}$ is a differential scheme: if $u \in I$, then differentiating Equation \ref{equation:linearization-definition} gives 
\begin{equation}\label{equation:differential-structure}
	L[u,\eta]' = L[u',\eta].
\end{equation}

\begin{remark}
Although we do not logically need this observation to prove our theorem, we can upgrade the isomorphism in Proposition~\ref{proposition:isomorphism-with-symmetric-algebra} to a morphism of differential rings using $D$-module structure on $\Omega_{A/K}$ and the $D$-module structure from \eqref{equation:differential-structure}.
The $D$-module structure on $\Omega_{A/K}$ is not new and appears in \cite{Johnson1969} and \cite[section 6]{Ollivier2022}.

Let $\Sigma = \Spec(A)$ where $A = K\lbrace x_1,\ldots,x_n\rbrace/I$ where $I$ is a differential ideal. 
The fact that $T\Sigma \to \Sigma$ is a morphism of $D$-schemes follows from Proposition~\ref{proposition:isomorphism-with-symmetric-algebra} and the fact that $\Sym(\Omega_{A/K})$ is a differential $K$-algebra. 
We identify $L[f]$ with $d f$ and use that $\partial(d f) = d(\partial(f))$ with \eqref{equation:differential-structure}.

Explicitly, in the case $I=0$ we have 
$$\Sym(\Omega_{K\lbrace x_1,\ldots,x_n \rbrace/K}) = K\lbrace x_1,\ldots,x_n \rbrace[dx_i^{(j)} : 1\leq i \leq n, j\geq 0].$$
The module $\Omega_{K\lbrace x_1,\ldots,x_n \rbrace/K}$ inherits the structure of a $D$-module where $\partial df = d \partial f$ for all $f \in K\lbrace x_1,\ldots,x_n\rbrace$.
In the case $\Sigma = \Spec K\lbrace x_1,\ldots,x_n\rbrace$ one has $T\Sigma = \Spec K\lbrace x_1,\ldots,x_n \rbrace \lbrace y_1,\ldots,y_n \rbrace$ where we can think of $y=(y_1,\ldots,y_n)$ as the extra variables appearing in the perturbation expansion $x+\varepsilon y$. 
The isomorphism $\Sym(\Omega_{K\lbrace x \rbrace/K}) \to K\lbrace x \rbrace \lbrace y \rbrace$ is given by $dx_i^{(j)} \mapsto y_i^{(j)}$ and it is an isomorphisms of $\partial$-algebras.

The case $I\neq 0$ can be made similarly explicit.
We have 
\begin{align*}\label{E:tangent-bundle-ideal}
	\Sym(\Omega_{A/K}) &= A[dx_i,d\dot{x}_i,d\ddot{x}_i,\ldots \colon 1\leq i \leq n]/[ du_i \colon 1\leq i \leq n ]\\
	&\cong K\lbrace x \rbrace \lbrace y \rbrace/\left[ I, \dfrac{\partial u}{\partial \nabla^{\infty} x} \cdot \nabla^{\infty}(y) \colon u \in I\right],
\end{align*}
and the isomorphism of $A$-algebras is given by $dx^{(r)}_i \mapsto y_i^{(r)}$.
One sees that this morphism passes to the level of quotients since $du = \dfrac{\partial u}{\partial \nabla^{\infty} x} \cdot d\nabla^{\infty}(x)$. 
This is visibly a bijection of ideals are and hence the morphism is an isomorphism of differential $A$-algebras. 
\end{remark}

\subsection{Linearization and Jacobi Bounds}

We now give two examples of differential schemes $\Sigma$ of finite Krull dimension and a point $\eta \in \Sigma$ such that $L[\Sigma,\eta]$ is infinite dimensional.
In Example~\ref{example:degenerate-linearization-01} the component is not generically reduced. In Example~\ref{example:degenerate-linearization-02} the point is not generic and the underlying component is generically reduced. 
\begin{example}\label{example:degenerate-linearization-01}
 Continuing with Example \ref{example:really-big-point}, let $A = K\{x\}/[x^2]$ and let $\eta$ be the generic point of the reduction $\Sigma_{\red} \cong \Spec K$ (given by setting each $x^{(i)} = 0$). 
 Then $L[x^2,\eta] = \frac{\partial u}{\partial x}(\eta) y = \left(2x \vert_{\eta} \right)y = (0)y = 0$, and $L[\Sigma]_{\eta} = \Spec K\lbrace y \rbrace$. In particular, the tangent bundle of $\Sigma$ is an infinite dimensional affine space.
\end{example}

\begin{example}\label{example:degenerate-linearization-02}
	Consider the system of equations in two differential variables $x,y$ given by $u_1=y^2-x^3=0$ and $u_2=x'=0$. 
	The Jacobi number of this system is one. 
	Also the scheme $\Spec K\lbrace x,y\rbrace/[ y^2-x^3, x']$ is irreducible but not reduced.
	To see this observe that $2yy'-3x^2x'=0$ implies $2yy'=0$ implies $yy'=0$.
	Taking another derivative gives $(y')^2 + yy''=0$.
	Multiplying by $y'$ gives $(y')^3=0$ (using the relation $yy'=0$). 
	After taking radicals we find that $y' \in \sqrt{ [y^2-x^3,x'] }$.
	Hence 
	$$K\lbrace x,y\rbrace/\sqrt{[y^2-x^3,x']} \cong K[x,y]/\langle y^2-x^3\rangle,$$ 
	which has Krull dimension one.
	
	After linearizing using $x=x_0+\varepsilon x_1$ and $y=y_0+\varepsilon y_1$ at $\eta=(x_0,y_0)=(0,0)$ we find that $L[u_1,\eta]=0$ and $L[u_2,\eta]=x_1'$. 
	This linearized system has infinite Krull dimension and differential dimension one. 
\end{example}

Just taking powers of differential polynomials is another straightforward way to produce examples of differential schemes which are not generically reduced. Indeed, if $I = [u_1,\ldots,u_n]$ is a differential ideal, then $J = [u_1^{a_1},\ldots,u_n^{a_1}]$ is a non-generically reduced differential ideal with $\sqrt{I} = \sqrt{J}$.) In such examples the non-reducedness is ``visible'', and its clear one should work with $I$ rather than $J$.
In other words, if one starts with the ideal $J$, then the radical is computed by first taking roots of the generators, and  $J(u_1,\ldots,u_n)  \leq J(u_1^{a_1},\ldots,u_n^{a_1})$.

In general, ``removing'' the non-reducedness is not straightforward, as the following example shows. 
\begin{example}\label{example:non-obvious}
  Let $J = \langle g_1, g_2, g_3 \rangle = \langle s^2 + tu, t^2 + us, u^2 + st\rangle  \subset K[s,t,u]$ and consider the (non differential) scheme $\Spec K[s,t,u]/J$. Projectively, these three conics do not intersect, so as an affine scheme this is a single non reduced point. Every element is nilpotent, but to exhibit this explicitly for e.g.~the variable $s$, one observes the syzygy
  \[
2s^4 = (2s^2 - tu)g_1 +u^2g_2 - sug_3 \in J;
\]
by construction no element of $J$ has degree 1, so $s \not \in J$.

From here, one can form a non-reduced differential scheme by substitution. For example, let $f_1,f_2,f_3 \in K\{x_1,x_2,x_3\}$ be any differential polynomials, and let $u_i = g_i(f_1,f_2,f_3)$. Then the differential scheme given by the ideal $I = [u_1,u_2,u_3]$ is not generically reduced. Moreover, to compute the radical of $I$, one first must acknowledge that $\sqrt{I} \supset [f_1,f_2,f_3]$. The $f_i$'s were arbitrary; so sometimes we will have $\sqrt{I} = [f_1,f_2,f_3]$, but often $\sqrt{I}$ will be larger and require more than three generators. This poses substantial problems:
  \begin{enumerate}
  \item if $\sqrt{I}$ is not minimally generated by three elements, we are no longer in the setting of the Jacobi bound conjecture; but worse,
  \item even if $\sqrt{I} = [f_1,f_2,f_3]$, one generally does not know any relationship between $J(f_1,f_2,f_3)$ and $J(u_1,u_2,u_3)$.
  \end{enumerate}
\end{example}

Below we use generic smoothness (Proposition~\ref{proposition:local-dimension}) to show that the sort of situation in Examples~\ref{example:degenerate-linearization-01}, \ref{example:degenerate-linearization-02}, and \ref{example:non-obvious} can't happen at generic points of generically reduced components. 
\\ 

First we show that, whether or not the linearization is infinite dimensional, its Jacobi number cannot increase.

\begin{proposition}[Jacobi Bound is non-increasing under linearization]\label{proposition:linearizing-jacobi-numbers}
	Let $\widehat{K}$ be a differential closure of $K$. Then the following are true.
	\begin{enumerate}
		\item Let $u\in K\lbrace x_1,\ldots,x_n \rbrace$. If $\eta\in \widehat{K}^n$ satisfies $u(\eta)=0$ then $\ord_{x_j}^{\partial} L[u,\eta] \leq \ord_{x_j}^{\partial}(u)$ for all $x_j$.
		\item Let $u_1,\ldots,u_n \in K\lbrace x_1,\ldots,x_n\rbrace$. If $\eta \in \widehat{K}^n$ satisfies $u_1(\eta)=u_2(\eta)=\cdots=u_n(\eta)=0$ then 
		$$ J(L[u_1,\eta], L[u_2,\eta],\ldots,L[u_n,\eta]) \leq J(u_1,\ldots,u_n).$$
		Furthermore, on a Kolchin open subset $U$ of $V([u_1,\ldots,u_n])$ we have 
		$$J(L[u_1,\eta], L[u_2,\eta],\ldots,L[u_n,\eta]) = J(u_1,\ldots,u_n),$$
		for all differential points $\eta \in U$.
	\end{enumerate}
\end{proposition}
\begin{proof}
	In what follows we let $\Max^{+}$ denote the maximum of a finite set and zero.
	For $K\lbrace x_1,\ldots,x_n \rbrace$ we have a formula for the order given by 
	$$\ord_{x_j}^{\partial}(u) = \Max^+\left\lbrace r \in \ZZ_{\geq 0} \colon \frac{\partial u}{\partial x_j^{(r)}} \neq 0 \right\rbrace.$$	
	From the formula for linearization \eqref{equation:linearized} we have 
	$$\ord_{y_j}^{\partial}(L[u,\eta]) = \Max^+\left\lbrace r \in \ZZ_{\geq 0} \colon \frac{\partial u}{\partial x_j^{(r)}}(\eta) \neq 0 \right\rbrace.$$
	This implies that $\ord^{\partial}_{x_j}(u) \geq \ord^{\partial}_{y_j}(L[u,\eta])$. 
	This holds for $1\leq j \leq n$. 
	We now have 
	\begin{align*}
		J(L[u_1,\eta], L[u_2,\eta],\ldots,L[u_n,\eta]) =& \Max_{\sigma \in S_n} \sum_{j=1}^n \ord^{\partial}_{y_j}( L[u_{\sigma(j)},\eta] )\\
		\leq & \Max_{\sigma \in S_n} \sum_{j=1}^n \ord^{\partial}_{x_j}( u_{\sigma(j)} ) \\
		=& J(u_1,\ldots,u_n).
	\end{align*}
	The Kolchin open subset where the Jacobi number doesn't drop after linearization is the collection of differential points $\eta \in V([u_1,\ldots,u_n])$ such that $\frac{\partial u_i}{\partial x_j^{(a_{ij})}}\neq 0$ where $a_{ij} = \ord_{x_j}^{\partial}(u_i)$.
\end{proof}

\section{Proof of Main Theorem}\label{section:main-proof}
We now prove Theorem~\ref{T:main}.
\\

Let $\Sigma = \Spec(K\lbrace x_1,\ldots,x_n\rbrace/[u_1,\ldots,u_n])$. 
Let $\Sigma_1 \subset \Sigma$ be a generically reduced component and let $\eta \in \Sigma_1$ be its generic point. 
From Section \ref{section:tangent-bundle}, equations \eqref{E:component-dimension} and \eqref{E:tangent-space-dimension}, 
since $\Sigma_1$ is generically smooth, we have 
$$
\Krull(T\Sigma_{\eta}) = \Krull(\Sigma_1).
$$
We have $T\Sigma_{\eta} \cong L[\Sigma,\eta]$ by Corollary~\ref{corollary:tangent-bundle-and-linearization}. 
Now $L[\Sigma,\eta]$ is a linear differential variety and by equation \eqref{equation:linear-equations} its defining ideal is $[L[u_1,\eta],\ldots,L[u_n,\eta]]$. 
$$
\Krull(L[\Sigma,\eta]) \leq J(L[u_1,\eta],\ldots,L[u_n,\eta]).
$$
By Proposition \ref{proposition:linearizing-jacobi-numbers}, $J(L[u_1,\eta],\ldots,L[u_n,\eta]) \leq J(u_1,\ldots,u_n)$, and we conclude that 
\[
\Krull(\Sigma_1) \leq J(u_1,\ldots,u_n).
\]

\section{Applications}\label{S:jbc-examples}

We now reprove \cite[Theorem 23]{Li2016} using our linearization technique. This is the Jacobi bound conjecture for characteristic sequences.
In this case we get equality.
\begin{corollary}\label{C:characteristic-sets}
	Let $R=K\lbrace x_1,\ldots,x_n\rbrace$.
	Let $P \subset R$ be a prime differential ideal with $\trdeg^{\partial}_K(R/P)<\infty$. 
	If $(g_1,\ldots,g_n)$ is a characteristic set for $P$ with respect to some ranking $\prec$, then 
	$$ \trdeg_K^{\partial}(R/P) = J(g_1,\ldots,g_n). $$
\end{corollary}
\begin{proof}
	Let $S$ denote the multiplicative set generated by the initials and separants of the $g_i$. 
	By definition of characteristic set we have that 
	\begin{equation}\label{E:characteristic-set}
		P=\sat_{S}([g_1,\ldots,g_n])
	\end{equation} 
	where $\sat$ denotes the saturation with respect to the multiplicative set $S$. 
	Let $\ell_{i}$ denote the leader of $g_i$. 
	Since $S \cap P=\emptyset$ we have $\partial g_i/\partial \ell_i \notin P$ and since the leading variables of all of the elements of the characteristic set are unique $[g_1,\ldots,g_n]$ has a good linearization at $P$. 
	
	Note that $J(g_1,\ldots,g_n)$ in this case is just the Ritt bound, i.e.~the maximum of each column, since every element is reduced with respect to every other element and the leading variables are unique. 
	
	Also the dimension is precisely the Jacobi Bound in this case \cite[Theorem 4.5 and discussion below]{Cluzeau2003}.
\end{proof}

Corollary~\ref{C:characteristic-sets} implies that for any given explicit set of differential equations $u_1,\ldots,u_n \in K\lbrace x_1,\ldots,x_n\rbrace$ one can effectively determine whether Jacobi Bound Conjecture holds via a computer algorithm.
\begin{corollary}\label{C:decidability}
	The problem of determining if the Jacobi Bound Conjecture holds for a given set of differential equations $(u_1,\ldots,u_n)$ in $K\lbrace x_1,\ldots,x_n\rbrace$ is decidable. 
\end{corollary}
\begin{proof}
	Given $[u_1,\ldots,u_n]$ we can run the Rosenfeld--Gr\"{o}bner algorithm to extract a prime decomposition of $\sqrt{[u_1,\ldots,u_n]}$.
	The output of the Rosenfeld--Gr\"{o}bner algorithm is a set of characteristic sets associated to $\sqrt{[u_1,\ldots,u_n]}$ --- each one of these defining a prime differential ideal as in \eqref{E:characteristic-set}. 
	Since for each characteristic set $(g_1,\ldots,g_n)$ associated to a prime ideal $P$ we have by Corollary~\ref{C:characteristic-sets} exactly $\trdeg^{\partial}_K(R/P)=J(g_1,\ldots,g_n)$ for $P$ finite dimensional (those characteristic sets which do not have finite Krull dimension will have length less than $n$) we can then compare $J(u_1,\ldots,u_n)$ explicitly to the finite number of $J(g_1,\ldots,g_n)$ to determine if the Jacobi Bound Conjecture holds in this example. 
\end{proof}

\providecommand{\bysame}{\leavevmode\hbox to3em{\hrulefill}\thinspace}
\providecommand{\MR}{\relax\ifhmode\unskip\space\fi MR }
\providecommand{\MRhref}[2]{%
  \href{http://www.ams.org/mathscinet-getitem?mr=#1}{#2}
}
\providecommand{\href}[2]{#2}

\end{document}